\documentclass{amsart}

\usepackage{amsfonts,amssymb,latexsym,amsmath, amsxtra, graphicx}
\usepackage{verbatim}

\newcommand{\field}[1]{\mathbb{#1}}

\newcommand{\N}{\field{N}}

\newcommand{\ov}{\overline}

\numberwithin{equation}{section}
\newtheorem{theorem}{Theorem}[section]
\newtheorem{lemma}[theorem]{Lemma}

\newtheorem{proposition}[theorem]{Proposition}
\theoremstyle{remark}

\renewenvironment{proof}[1][Proof]{\begin{trivlist}
\item[\hskip \labelsep {\bfseries #1:}]}{\qed\end{trivlist}}

\title[A generalisation of a second partition theorem of Andrews]{A generalisation of a second partition theorem of Andrews to overpartitions}
\author{Jehanne Dousse}
\address{LIAFA \\
Universite Paris Diderot - Paris 7 \\
75205 Paris Cedex 13 \\
FRANCE }
\email{jehanne.dousse@liafa.univ-paris-diderot.fr}
\allowdisplaybreaks

\begin{document}

\begin{abstract}
In 1968 and 1969, Andrews proved two partition theorems of the Rogers-Ramanujan type which generalise Schur's celebrated partition identity (1926). Andrews' two generalisations of Schur's theorem went on to become two of the most influential results in the theory of partitions, finding applications in combinatorics, representation theory and quantum algebra. In a recent paper, the author generalised the first of these theorems to overpartitions, using a new technique which consists in going back and forth between $q$-difference equations on generating functions and recurrence equations on their coefficients. Here, using a similar method, we generalise the second theorem of Andrews to overpartitions.
\end{abstract}

\maketitle

%
%

\section{Introduction}
A partition of $n$ is a non-increasing sequence of natural numbers whose sum is $n$.
An overpartition of $n$ is a partition of $n$ in which the first occurrence of a number may be overlined.
For example, there are $14$ overpartitions of $4$:
$4$, $\overline{4}$, $3+1$, $\overline{3}+1$, $3+\overline{1}$, $\overline{3}+\overline{1}$, $2+2$, $\overline{2}+2$, $2+1+1$, $\overline{2}+1+1$, $2+\overline{1}+1$, $\overline{2}+\overline{1}+1$, $1+1+1+1$ and $\overline{1}+1+1+1$.

In 1926, Schur~\cite{Schur} proved the following partition identity.

\begin{theorem}[Schur]
\label{schur}
Let $n$ be a positive integer.
Let $D_1(n)$ denote the number of partitions of $n$ into distinct parts congruent to $1$ or $2$ modulo $3$.
Let $E_1(n)$ denote the number of partitions of $n$ of the form $n= \lambda_1 + \cdots + \lambda_s$ where $\lambda_i - \lambda_{i+1} \geq 3$ with strict inequality if $\lambda_{i+1} \equiv 0 \mod 3$.
Then $D_1(n)=E_1(n)$.
\end{theorem}
For example, for $n=9$, the partitions counted by $D_1(9)$ are $8+1$, $7+2$ and $5+4$ and the partitions counted by $E_1(9)$ are $9$, $8+1$ and $7+2$. Thus $D_1(9)=E_1(9)=3.$

Several proofs of Schur's theorem have been given using a variety of different techniques such as bijective mappings~\cite{Bessenrodt,Bressoud}, the method of weighted words~\cite{Alladi}, and recurrences~\cite{Andrews2,Andrews1,Andrews3}.

Schur's theorem was subsequently generalised to overpartitions by Lovejoy~\cite{Lovejoy} using the method of weighted words. The case $k=0$ corresponds to Schur's theorem.

\begin{theorem}[Lovejoy]
\label{schur_over}
Let $D_1(k,n)$ denote the number of overpartitions of $n$ into parts congruent to $1$ or $2$ modulo $3$ with $k$ non-overlined parts.
Let $E_1(k,n)$ denote the number of overpartitions of $n$ with $k$ non-overlined parts, where parts differ by at least $3$ if the smaller is overlined or both parts are divisible by $3$, and parts differ by at least $6$ if the smaller is overlined and both parts are divisible by $3$.
Then $D_1(k,n)=E_1(k,n)$.
\end{theorem}

Theorem~\ref{schur_over} was then proved bijectively by Raghavendra and Padmavathamma \cite{Pad}, and using $q$-difference equations and recurrences by the author~\cite{Dousse}.

Andrews extended the ideas of his proofs of Schur's theorem to prove two much more general theorems on partitions with difference conditions~\cite{Generalisation2,Generalisation1}. But before stating these results in their full generality we need to introduce some notation.
Let $A=\lbrace a(1), ..., a(r) \rbrace$ be a set of $r$ distinct integers such that $\sum_{i=1}^{k-1} a(i) < a(k)$ for all $1 \leq k \leq r$ and the $2^r -1$ possible sums of distinct elements of $A$ are all distinct. We denote this set of sums by $A'=\lbrace \alpha(1), ..., \alpha(2^r -1) \rbrace$, where $\alpha(1) < \cdots < \alpha(2^r-1)$. Let us notice that $\alpha(2^k)=a(k+1)$ for all $0 \leq k \leq r-1$ and that any $\alpha$ between $a(k)$ and $a(k+1)$ has largest summand $a(k)$.
Let $N$ be a positive integer with $N \geq \alpha(2^r-1) = a(1) +\cdots+a(r).$ We further define $\alpha(2^r)=a(r+1)=N+a(1).$ Let $A_N$ denote the set of positive integers congruent to some $a(i) \mod N$, $-A_N$ the set of positive integers congruent to some $-a(i) \mod N$, $A'_N$ the set of positive integers congruent to some $\alpha(i) \mod N$ and $-A'_N$ the set of positive integers congruent to some $-\alpha(i) \mod N.$ Let $\beta_N(m)$ be the least positive residue of $m \mod N$. If $\alpha \in A'$, let $w(\alpha)$ be the number of terms appearing in the defining sum of $\alpha$ and $v(\alpha)$ the smallest $a(i)$ appearing in this sum.

To illustrate these notations in the remainder of this paper, it might be useful to consider the example where $a(k)=2^{k-1}$ for $1 \leq k \leq r$ and $\alpha(k)=k$ for $1 \leq k \leq 2^r-1$.

We are now able to state Andrews' generalisations of Schur's theorem.

\begin{theorem}[Andrews]
\label{andrews}
Let $D(A_N;n)$ denote the number of partitions of $n$ into distinct parts taken from $A_N$. Let $E(A'_N;n)$ denote the number of partitions of $n$ into parts taken from $A'_N$ of the form $n=\lambda_1+\cdots+ \lambda_s$, such that
\[\lambda_i - \lambda_{i+1} \geq N w(\beta_N(\lambda_{i+1}))+v(\beta_N(\lambda_{i+1}))-\beta_N(\lambda_{i+1}).\]
Then $D(A_N;n)= E(A'_N;n)$.
\end{theorem}

\begin{theorem}[Andrews]
\label{andrews2}
Let $F(-A_N;n)$ denote the number of partitions of $n$ into distinct parts taken from $-A_N$. Let $G(-A'_N;n)$ denote the number of partitions of $n$ into parts taken from $-A'_N$ of the form $n=\lambda_1+\cdots+ \lambda_s$, such that
\begin{displaymath}
\lambda_i - \lambda_{i+1} \geq N w(\beta_N(-\lambda_{i+1}))+v(\beta_N(-\lambda_{i+1}))-\beta_N(-\lambda_{i+1}),
\end{displaymath}
and
\begin{math} \lambda_s \geq N (w(\beta_N(-\lambda_s)-1). \end{math}
Then $F(-A_N;n)= G(-A'_N;n)$.
\end{theorem}

Not only have Andrews' identities led to a number of important developments in combinatorics~\cite{Alladi1,Corteel,Yee} but they also play a natural role in group representation theory~\cite{AndrewsOlsson} and quantum algebra~\cite{Oh}.

The author generalised Theorem~\ref{andrews} to overpartitions in~\cite{Doussegene} by proving the following.

\begin{theorem}
\label{dousse}
Let $D(A_N;k,n)$ denote the number of overpartitions of $n$ into parts taken from $A_N$, having $k$ non-overlined parts. Let $E(A'_N;k,n)$ denote the number of overpartitions of $n$ into parts taken from $A'_N$ of the form $n=\lambda_1+\cdots+ \lambda_s$, having $k$ non-overlined parts, such that
\[\lambda_i - \lambda_{i+1} \geq N w\left(\beta_N(\lambda_{i+1}) -1 +\chi(\ov{\lambda_{i+1}}) \right)+v(\beta_N(\lambda_{i+1}))-\beta_N(\lambda_{i+1}),\]
where $\chi(\ov{\lambda_{i+1}})=1$ if $\lambda_{i+1}$ is overlined and $0$ otherwise.

Then $D(A_N;k,n)= E(A'_N;k,n)$.
\end{theorem}

Here we generalise Theorem~\ref{andrews2} by showing the following.

\begin{theorem}
\label{dousse2}
Let $F(-A_N;k,n)$ denote the number of overpartitions of $n$ into parts taken from $-A_N$, having $k$ non-overlined parts. Let $G(-A'_N;k,n)$ denote the number of overpartitions of $n$ into parts taken from $-A'_N$ of the form $n=\lambda_1+\cdots+ \lambda_s$, having $k$ non-overlined parts, such that
\begin{displaymath}
\lambda_i - \lambda_{i+1} \geq N w\left(\beta_N(\lambda_{i+1}) -1 +\chi(\ov{\lambda_{i+1}}) \right)+v(\beta_N(\lambda_{i+1}))-\beta_N(\lambda_{i+1}),
\end{displaymath}
\begin{displaymath}
\lambda_s \geq N (w(\beta_N(-\lambda_s))-1).
\end{displaymath}

Then $F(-A_N;k,n)= G(-A'_N;k,n)$.
\end{theorem}

Theorem~\ref{schur} (resp. Theorem~\ref{schur_over}) corresponds to $N=3$, $a(1)= 1$, $a(2)=2$ in Theorems~\ref{andrews} and~\ref{andrews2} (resp. Theorems~\ref{dousse} and~\ref{dousse2}).
Again, the case $k=0$ of Theorem~\ref{dousse} (resp. Theorem~\ref{dousse2}) gives Theorem~\ref{andrews} (resp. Theorem~\ref{andrews2}).

Let us illustrate Theorems \ref{dousse2} with an example. Let $N=7$, $r=3$, $a(1)=1,$ $a(2)=2,$ $a(3)=4$.
The overpartitions of $8$ counted by $G(-A'_7;k,8)$ are $8$, $\overline{8}$, $5+3$ and $\overline{5}+3$. The overpartitions of $8$ into parts congruent to $3$, $5$ or $6$ modulo $7$ (counted by $F(-A_7;k,8)$) are $5+3$, $\overline{5}+3$, $5 + \overline{3}$ and $\overline{5}+ \overline{3}$. In both cases, we have $1$ overpartition with $0$ non-overlined parts, $2$ overpartitions with $1$ non-overlined part, and $1$ overpartition with $2$ non-overlined parts.

While the statements of Theorems~\ref{dousse} and~\ref{dousse2} resemble those of Andrews' theorems, the proofs are considerably more intricate and involve a number of new ideas. The proof of Theorem~\ref{dousse2} uses ideas similar to the proof of Theorem~\ref{dousse} presented in~\cite{Doussegene}, in the sense that it relies in going back and forth from $q$-difference equations to recurrence equations.

The remainder of this paper is devoted to the proof of Theorem~\ref{dousse2}.
First, we give the recurrence equation satisfied by the generating function for overpartitions enumerated by $G(-A'_N;k,n)$ having their largest part $\leq m$, using some combinatorial reasoning on the largest part. Then we prove by induction on $r$ that the limit when $m$ goes to infinity of a function satisfying this recurrence equation is equal to \begin{math} \prod_{j=1}^r \frac{(-q^{N-a(j)};q^N)_{\infty}}{(dq^{N-a(j)};q^N)_{\infty}}, \end{math} which is the generating function for overpartitions counted by $F(-A_N;k,n)$. Here we use the classical notation $(a;q)_{n} := \prod_{j=0}^{n-1} (1-aq^j).$

\section{The recurrence equation}
In this section, we establish the recurrence equation satisfied by the generating function for overpartitions enumerated by $G(-A'_N;k,n)$ having their largest part $\leq m$.

Let $n, m \in \N^*$, $k \in \N$.
Let $\pi_m(k,n)$ denote the number of overpartitions counted by $G(-A'_N;k,n)$ such that the largest part is $ \leq m$ and overlined.
Let $\phi_m(k,n)$ denote the number of overpartitions counted by $G(-A'_N;k,n)$ such that the largest part is $ \leq m$ and non-overlined.
Then $\psi_{m} (k,n) := \pi_{m}(k,n) + \phi_{m}(k,n)$ is the number of overpartitions counted by $G(-A'_N;k,n)$ with largest part $\leq m$.

Then the following holds.

\begin{lemma}
\label{lemma1}
We have
\begin{equation}
\label{eq1}
\begin{aligned}
\psi_{jN-\alpha(m)} (k,n) &- \psi_{jN-\alpha(m+1)} (k,n)
\\ =& ~\psi_{jN-w(\alpha(m))N-v(\alpha(m))} (k,n-jN+\alpha(m))\\
+& \psi_{jN-(w(\alpha(m))-1)N-v(\alpha(m))} (k,n-jN+\alpha(m)).
\end{aligned}
\end{equation}
\end{lemma}
\begin{proof}
Let us first prove the following equation:
\begin{equation}
\label{aux1}
\begin{aligned}
\pi_{jN-\alpha(m)} (k,n) &= \pi_{jN-\alpha(m+1)} (k,n)
\\ &+ \pi_{jN-w(\alpha(m))N-v(\alpha(m))} (k,n-jN+\alpha(m))\\
&+ \phi_{jN-(w(\alpha(m))-1)N-v(\alpha(m))} (k,n-jN+\alpha(m)).
\end{aligned}
\end{equation}

We break the overpartitions counted by $\pi_{jN-\alpha(m)} (k,n)$ into two sets : those with largest part \begin{math} < jN-\alpha(m) \end{math} and those with largest part equal to $jN-\alpha(m)$. The first set is counted by $\pi_{jN-\alpha(m+1)} (k,n)$, and the second by
\begin{align*}
&\pi_{jN-w(\alpha(m))N-v(\alpha(m))} (k,n-jN+a(m))
\\&+ \phi_{jN-(w(\alpha(m))-1)N-v(\alpha(m))}(k,n-jN+a(m)).
\end{align*}
To see this, let us consider an overparition $n=\lambda_1 + \lambda_2 + \dots + \lambda_s$ counted by \begin{math}\pi_{jN-\alpha(m)} (k,n)\end{math} with largest part equal to \begin{math}jN-\alpha(m)\end{math}. Now remove its largest part $\lambda_1 = jN-\alpha(m)$. The number partitioned becomes  $n-jN+\alpha(m)$. The largest part was overlined so the number of non-overlined parts is still $k$. If $\lambda_2$ was overlined, then we have
\begin{align*}
\lambda_2 &\leq \lambda_1 - w(\alpha(m))N -v(\alpha(m)) +\alpha(m)\\
&\leq jN - w(\alpha(m))N -v(\alpha(m)),
\end{align*}
and we obtain an overpartition counted by $\pi_{jN-w(\alpha(m))N-v(\alpha(m))} (k,n-jN+a(m))$.
If $\lambda_2$ was not overlined, then we have
\begin{align*}
\lambda_2 &\leq \lambda_1 - (w(\alpha(m))-1)N -v(\alpha(m)) +\alpha(m)\\
&\leq jN - (w(\alpha(m))-1)N -v(\alpha(m)),
\end{align*}
and we obtain an overpartition counted by $\phi_{jN-(w(\alpha(m))-1)N-v(\alpha(m))} (k,n-jN+a(m))$.

In the same way we can prove the following
\begin{equation}
\label{aux2}
\begin{aligned}
\phi_{jN-\alpha(m)} (k,n) &= \phi_{jN-\alpha(m+1)} (k,n)
\\ &+ \pi_{jN-w(\alpha(m))N-v(\alpha(m))} (k-1,n-jN+\alpha(m))\\
&+ \phi_{jN-(w(\alpha(m))-1)N-v(\alpha(m))} (k-1,n-jN+\alpha(m)).
\end{aligned}
\end{equation}

Adding equations~\eqref{aux1} and~\eqref{aux2} and noting that for all $m,n,k,$ $\pi_m(k-1,n) = \phi_m(k,n)$ (we can either overline the largest part or not), we obtain equation~\eqref{eq1}.
\end{proof}

We define, for $m \geq 1$, $|q|<1$, $|d|<1$,
\begin{equation*}
g_m=g_m (q,d) := 1+ \sum_{n=1}^{\infty} \sum_{k=0}^{\infty} \psi_{m}(k,n) q^n d^k,
\end{equation*}
and for all $0 \leq k \leq r-1$, we set $g_{-m} (q,d)=(-d)^k$ for all $kN \leq m \leq (k+1)N$.
This definition is consistent with~\eqref{eq1} and the condition that \begin{math} \lambda_s \geq N (w(\beta_N(-\lambda_s))-1) \end{math}.

We want to find $\lim\limits_{m \rightarrow \infty} g_m$, which is the generating function for all overpartitions counted by $G(-A'_N;k,n)$. To do so, we  establish a recurrence equation relating $g_{(m-j)N-a(1)}$, for $0 \leq j \leq r$. Let us start by giving some relations between generating functions.

Lemma~\ref{lemma1} directly implies

\begin{lemma}
\label{lemma2}
We have
\begin{equation}
\label{eqf1}
\begin{aligned}
g_{jN-\alpha(m)}= g_{jN-\alpha(m+1)} +& q^{jN-\alpha(m)} g_{jN-w(\alpha(m))N-v(\alpha(m))}\\
+& dq^{jN-\alpha(m)} g_{jN-(w(\alpha(m))-1)N-v(\alpha(m))}.
\end{aligned}
\end{equation}

\end{lemma}
\noindent Let $ 1 \leq k \leq r+1.$ Adding equations~\eqref{eqf1} together for $1 \leq m \leq 2^{k-1}-1$, using the fact that $\alpha\left(2^{k-1}\right)=a(k)$, we obtain

\begin{equation}
\label{eq3.5}
\begin{aligned}
g_{jN- a(1)} &= g_{jN -a(k)} 
\\&+ \sum_{\alpha < a(k)} \left( q^{jN-\alpha} g_{(j-w(\alpha))N-v(\alpha)} + dq^{jN-\alpha}g_{(j-w(\alpha)+1)N-v(\alpha)}\right).
\end{aligned}
\end{equation}
Let us now add equations~\eqref{eqf1} together for $2^{k-2} \leq m \leq 2^{k-1}-1$. This gives

\begin{equation}
\label{eq3.6}
\begin{aligned}
g_{jN- a(k)} &= g_{jN -a(k+1)} 
\\&+ \sum_{a(k) \leq \alpha < a(k+1)} \left( q^{jN-\alpha} g_{(j-w(\alpha))N-v(\alpha)} + dq^{jN-\alpha}g_{(j-w(\alpha)+1)N-v(\alpha)}\right).
\end{aligned}
\end{equation}
Every $a(k) < \alpha < a(k+1)$ is of the form $\alpha = a(k) + \alpha',$ with $\alpha' < a(k).$
Hence we can rewrite~\eqref{eq3.6} as

\begin{align*}
&g_{jN- a(k)} - g_{jN -a(k+1)}
\\&= q^{jN-a(k)} g_{(j-1)N-a(k)} + dq^{jN-a(k)} g_{jN-a(k)}
\\&+ \sum_{\alpha' < a(k)} \left( q^{jN-a(k)-\alpha'}  g_{(j-w(\alpha')-1)N-v(\alpha')}  + dq^{jN-a(k)-\alpha'}  g_{(j-w(\alpha'))N-v(\alpha')}\right)
\\&= q^{jN-a(k)} g_{(j-1)N-a(k)} + dq^{jN-a(k)} g_{jN-a(k)}
\\&+ q^{N-a(k)} \left(g_{(j-1)N- a(1)} - g_{(j-1)N -a(k)}\right),
\end{align*}
where the last equality follows from~\eqref{eq3.5}.

Thus
\begin{equation}
\label{eq3.7}
\begin{aligned}
\left(1-q^{jN-a(k)}\right) g_{jN-a(k)} &= g_{jN-a(k+1)}+q^{N-a(k)}g_{(j-1)N-a(1)}
\\&-q^{N-a(k)}\left(1-q^{(j-1)N}\right)g_{(j-1)N-a(k)}.
\end{aligned}
\end{equation}

We want to find the recurrence equation satisfied by $(g_{\ell N-a(1)})_{\ell \in \N}$.
Before doing so, we must recall some facts about $q$-binomial coefficients defined by
$${m \brack r}_q :=
\begin{cases}
\frac{\left(1-q^m\right)\left(1-q^{m-1}\right) \dots \left(1-q^{m-r+1}\right)}{\left(1-q\right) \left(1-q^2\right) \dots \left(1-q^r\right)}\ \text{if}\ 0 \leq r \leq m,\\
0 \ \text{otherwise}.
\end{cases}$$
They are $q$-analogues of the binomial coefficients and satisfy $q$-analogues of the Pascal triangle identity~\cite{Gasper}.

\begin{proposition}
\label{pascal}
For all integers $0 \leq r \leq m$,
\begin{equation}
\label{pascal1}
{m \brack r}_q = q^r {m-1 \brack r}_q + {m-1 \brack r-1}_q,
\end{equation}
\begin{equation}
\label{pascal2}
{m \brack r}_q ={m-1 \brack r}_q + q^{m-r} {m-1 \brack r-1}_q.
\end{equation}
\end{proposition}
As $q \rightarrow 1$ these equations become Pascal's identity.

We are now ready to state the key lemma which will lead to the desired recurrence equation.
\begin{lemma}
\label{conj}
For $1 \leq k \leq r+1$, we have
\begin{equation}
\label{eq}
\begin{aligned}
\prod_{j=1}^{k-1} &\left(1-dq^{\ell N-a(j)}\right) g_{\ell N-a(1)} = g_{\ell N-a(k)}
\\ + \sum_{j=1}^{k-1} &\left( \sum_{m=0}^{k-j-1} d^m \sum_{\substack{\alpha < a(k) \\ w(\alpha)=j+m}} q^{\ell N-\alpha} \left( (-1)^{m-1} q^{\ell (m-1)N} {j+m-1 \brack m-1}_{q^{-N}} \right. \right.
\\ &\left. \vphantom{\sum_{\substack{\alpha < a(r) \\ w(\alpha)=j+m-1}}} \left. + (-1)^{m} q^{\ell mN} {j+m \brack m}_{q^{-N}} \right) \right) \prod_{h=1}^{j-1} \left(1-q^{(\ell-h)N}\right) g_{(\ell-j) N-a(1)}.
\end{aligned}
\end{equation}
\end{lemma}
\begin{proof}
To prove this lemma, it is sufficient to replace $q$ by $q^{-1}$, then $x$ by $q^\ell N$ and finally $f_{a(i)}\left(q^{mN}\right)$ by $g_{mN-a(i)}$ in the proof of Lemma 2.4 of~\cite{Doussegene}.
\end{proof}

Writing $u_{\ell}:= g_{\ell N -a(1)}$ and setting $k=r+1$ in Lemma~\ref{conj}, we obtain the desired recurrence equation
\begin{equation}
\label{rec}
\tag{$\mathrm{rec}_{N,r}$}
\begin{aligned}
\prod_{j=1}^{r} &\left(1-dq^{\ell N-a(j)}\right) u_{\ell} = u_{\ell-1}
\\ + \sum_{j=1}^{r} &\left( \sum_{m=0}^{r-j} d^m \sum_{\substack{\alpha < a(r+1) \\ w(\alpha)=j+m}} q^{\ell N-\alpha} \left( (-1)^{m-1} q^{\ell (m-1)N} {j+m-1 \brack m-1}_{q^{-N}} \right. \right.
\\& \left. \vphantom{\sum_{\substack{\alpha < a(r) \\ w(\alpha)=j+m-1}}} \left. + (-1)^{m} q^{\ell mN} {j+m \brack m}_{q^{-N}} \right) \right) \prod_{h=1}^{j-1} \left(1-q^{(\ell-h)N}\right) u_{\ell-j},
\end{aligned}
\end{equation}
with the initial conditions $u_{-k}=(-d)^k$ for all $0 \leq k \leq r-1$.

\section{Evaluating $\lim\limits_{\ell \rightarrow \infty} u_{\ell}$ by induction}
In this section, we evaluate $\lim\limits_{\ell \rightarrow \infty} u_{\ell}$, which is the generating function for partitions counted by $G(-A'_N;k,n)$.
To do so, we prove the following theorem by induction on $r$.
\begin{theorem}
\label{main}
Let $r$ be a positive integer. Then for every $N \geq \alpha(2^r-1)$, for every sequence $(u_m)_{m \in \N}$ satisfying $(\mathrm{rec}_{N,r})$ and the initial condition $u_0=1$, we have 
\begin{displaymath}
\lim\limits_{\ell \rightarrow \infty} u_{\ell}= \prod_{k=1}^r \frac{(-q^{N-a(k)};q^N)_{\infty}}{(dq^{N-a(k)};q^N)_{\infty}}.
\end{displaymath} 
\end{theorem}

The idea of the proof is to start from a function satisfying $(\mathrm{rec}_{N,r})$ and to do some transformations to obtain a function satisfying $(\mathrm{rec}_{N,r-1})$ in order to use the induction hypothesis.
In order to simplify the proof, we split it into several lemmas.

\begin{lemma}
\label{lemmau}
Let $(u_m)$ and $(\beta_m)$ be two sequences such that for all $m \in \N$,
$$\beta_m:= u_m \prod_{j=1}^{m} \frac{1-dq^{jN-a(r)}}{1-q^{jN}}$$
Then $u_0=1$ and $(u_m)$ satisfies~\eqref{rec} if and only if $\beta_0=1$ and $(\beta_m)$ satisfies the following recurrence equation
\begin{equation}
\label{recbeta}
\tag{$\mathrm{rec}'_{N,r}$}
\begin{aligned}
&\left(1 + \sum_{j=1}^{r} \left( d^{j-1} \sum_{\substack{ \alpha < a(r) \\ w(\alpha)=j-1}} q^{-\alpha} +d^j \sum_{\substack{ \alpha < a(r) \\ w(\alpha)=j}} q^{-\alpha} \right) (-1)^j q^{j\ell N} \right) \beta_{\ell}
\\&= \beta_{\ell-1} + \sum_{j=1}^r \sum_{h=1}^r \sum_{k=0}^{\min(j-1,h-1)}c_{k,j}b_{h-k,j} (-1)^{h+1} q^{h \ell N} \beta_{\ell-j},
\end{aligned}
\end{equation}
where
$$c_{k,j}:= q^{-N \frac{k(k+1)}{2} - k a(r)} {j-1 \brack k}_{q^{-N}} d^k,$$
and
$$b_{m,j}:= \left( d^{m-1} \sum_{\substack{ \alpha < a(r+1) \\ w(\alpha)=j+m-1}} q^{-\alpha} +d^m \sum_{\substack{ \alpha < a(r+1) \\ w(\alpha)=j+m}} q^{-\alpha} \right) {j+m-1 \brack m-1}_{q^{-N}}.$$
\end{lemma}
\begin{proof}
Directly plugging the definition of $(\beta_m)$ into $(\mathrm{rec}_{N,r})$, we get
\begin{align*}
(1-q^{\ell N}) &\prod_{j=1}^{r-1} \left(1-dq^{\ell N-a(j)}\right) \beta_{\ell} = \beta_{\ell-1}
\\ + \sum_{j=1}^{r} &\left( \sum_{m=0}^{r-j} d^m \sum_{\substack{\alpha < a(r+1) \\ w(\alpha)=j+m}} q^{\ell N -\alpha} \left( (-1)^{m-1} q^{\ell (m-1) N} {j+m-1 \brack m-1}_{q^{-N}} \right.\right.
\\&\left. \vphantom{\sum_{\substack{\alpha < a(r) \\ w(\alpha)=j+m-1}}} \left. + (-1)^m q^{\ell m N} {j+m \brack m}_{q^{-N}} \right) \right) \prod_{h=1}^{j-1} \left(1-dq^{(\ell-h)N-a(r)}\right) \beta_{\ell-j}.
\end{align*}
With the conventions that
$$ \sum_{\substack{\alpha < a(r) \\ w(\alpha)=n}} q^{-\alpha} = 0 \ \text{for}\ n \geq r,$$
and
$$ \sum_{\substack{\alpha < a(r) \\ w(\alpha)=0}} q^{-\alpha} = 1,$$
this can be reformulated as
\begin{align*}
&\left(1 + \sum_{j=1}^{r} \left( d^{j-1} \sum_{\substack{ \alpha < a(r) \\ w(\alpha)=j-1}} q^{-\alpha} +d^j \sum_{\substack{ \alpha < a(r) \\ w(\alpha)=j}} q^{-\alpha} \right) (-1)^j q^{j\ell N} \right) \beta_{\ell} = \beta_{\ell -1}
\\ &+ \sum_{j=1}^{r} \left( \sum_{m=1}^{r-j+1} \left(d^{m-1} \sum_{\substack{\alpha < a(r+1) \\ w(\alpha)=j+m-1}} q^{-\alpha} + d^m \sum_{\substack{\alpha < a(r+1) \\ w(\alpha)=j+m}} q^{-\alpha}\right) {j+m-1 \brack m-1}_{q^{-N}} \right. 
\\ &\qquad \qquad \left. \vphantom{\sum_{\substack{\alpha < a(r) \\ w(\alpha)=j+m-1}}} (-1)^{m-1} q^{m \ell N} \right) \left(\sum_{k=0}^{j-1} q^{-N \frac{k(k-1)}{2}-k a(r)} {j-1 \brack k}_{q^{-N}} d^k (-1)^k q^{k \ell N}\right) \beta_{\ell-j},
\end{align*}
because of the $q$-binomial theorem~\cite{Gasper}
\begin{equation}
\label{identity}
\prod_{k=0}^{n-1}(1+q^kt) = \sum_{k=0}^n q^{\frac{k(k-1)}{2}} {n \brack k}_q t^k,
\end{equation}
in which we replace $q$ by $q^{-N}$, $n$ by $j-1$ and $t$ by $-dq^{(\ell-1)N-a(r)}$.
Finally, noting that $b_{l-k,j} = 0$ if $j+l-k-1 \geq r$, we can rewrite this as~\eqref{recbeta}.
Moreover, $\beta_0=u_0=1$ and the lemma is proved.
\end{proof}

We can directly transform ~\eqref{recbeta} into a $q$-difference equation on the generating function for $(\beta_m)$.
\begin{lemma}
\label{lemmaf}
Let $(\beta_m)$ be a sequence and $f$ a function such that $$f(x) := \sum_{n=0}^{\infty} \beta_n x^n.$$
Then $(\beta_m)$ satisfies \eqref{recbeta} and the initial condition $\beta_0=1$ if and only if $f(0)=1$ and $f$ satisfies the following recurrence equation
\begin{equation}
\label{eqf}
\tag{$\mathrm{eq}_{N,r}$}
\begin{aligned}
(1-x) f(x) = \sum_{m=1}^{r} \Bigg( &d^{m-1} \sum_{\substack{\alpha < a(r) \\ w(\alpha)=m-1}} q^{-\alpha} + d^{m} \sum_{\substack{\alpha < a(r) \\ w(\alpha)=m}} q^{\alpha}
\\&+ \sum_{j=1}^r \sum_{k=0}^{\min(j-1,m-1)} c_{k,j} b_{m-k,j} x^j q^{mjN} \Bigg) (-1)^{m+1} f(xq^{mN}).
\end{aligned}
\end{equation}
\end{lemma}
\begin{proof}
By the definition of $f$ and~\eqref{recbeta}, we have
\begin{align*}
(1-x) f(x) =& \sum_{j=1}^{r} \left( d^{j-1} \sum_{\substack{\alpha < a(r) \\ w(\alpha)=j-1}} q^{-\alpha} + d^{j} \sum_{\substack{\alpha < a(r) \\ w(\alpha)=j}} q^{-\alpha} \right) (-1)^{j+1} f\left(xq^{jN}\right)
\\ &+  \sum_{j=1}^r \sum_{h=1}^r \sum_{k=0}^{\min(j-1,h-1)} c_{k,j} b_{h-k,j} (-1)^{h+1} x^j q^{hjN} f\left(xq^{hN}\right).
\end{align*}
Relabelling the summation indices and factorising leads to~\eqref{eqf}. Moreover, $f(0)=\beta_0=1$. This completes the proof.
\end{proof}

Let us now do some transformations starting from $(\mathrm{rec}_{N,r-1})$.

\begin{lemma}
\label{lemmasn}
Let $(\mu_n)$ and $(s_n)$ be two sequences such that for all $n$,
$$s_n:= \mu_n \prod_{k=1}^{n} \frac{1}{1-q^{Nk}}.$$
Then $(\mu_n)$ satisfies $(\mathrm{rec}_{N,r-1})$ and the initial condition $\mu_0=1$ if and only if $s_0=1$ and $(s_n)$ satisfies the following recurrence equation
\begin{equation}
\label{recsn}
\tag{$\mathrm{rec}''_{N,r-1}$}
\begin{aligned}
&\left(1 + \sum_{j=1}^{r} \left( d^{j-1} \sum_{\substack{ \alpha < a(r) \\ w(\alpha)=j-1}} q^{-\alpha} +d^j \sum_{\substack{ \alpha < a(r) \\ w(\alpha)=j}} q^{-\alpha} \right) (-1)^j q^{j\ell N} \right) s_{\ell} = s_{\ell-1}
\\&+ \sum_{j=1}^r \sum_{m=1}^{r-j} \left(d^{m-1} \sum_{\substack{ \alpha < a(r) \\ w(\alpha)=j+m-1}} q^{-\alpha} +d^m \sum_{\substack{ \alpha < a(r) \\ w(\alpha)=j+m}} q^{-\alpha} \right)
\\& \qquad \qquad \quad \times {j+m-1 \brack m-1}_{q^{-N}} (-1)^{m+1} q^{m\ell N} s_{\ell-j}.
\end{aligned}
\end{equation}
\end{lemma}
\begin{proof}
Using the definition of $(s_n)$ and $(\mathrm{rec}_{N,r-1})$, we get
\begin{align*}
&(1-q^{\ell N}) \prod_{j=1}^{r-1} \left(1-dq^{\ell N-a(j)}\right) s_{\ell} = s_{\ell-1}
\\ + \sum_{j=1}^{r-1} &\left( \sum_{m=0}^{r-j-1} d^m \sum_{\substack{\alpha < a(r+1) \\ w(\alpha)=j+m}} q^{\ell N -\alpha} \left( (-1)^{m-1} q^{\ell (m-1) N} {j+m-1 \brack m-1}_{q^{-N}} \right. \right.
\\& \left. \vphantom{\sum_{\substack{\alpha < a(r) \\ w(\alpha)=j+m-1}}} \left. + (-1)^m q^{\ell m N} {j+m \brack m}_{q^{-N}} \right) \right) \prod_{h=1}^{j-1} \left(1-dq^{(\ell-h)N-a(r)}\right) s_{\ell-j}.
\end{align*}
Then, as in the proof of Lemma~\ref{lemmau}, this can be reformulated as~\eqref{recsn}, and $s_0= \mu_0=1.$
\end{proof}

Again, let us translate this into a recurrence equation on the generating function for $(s_n)$.

\begin{lemma}
\label{lemmaG}
Let $(s_n)$ be a sequence and $G$ be a function such that
$$G(x) := \sum_{n=0}^{\infty} s_n x^n.$$
Then $(s_n)$ satisfies $(\mathrm{rec}''_{N,r-1})$ and the initial condition $s_0=1$ if and only if $G(0)=1$ and $G$ satisfies the following $q$-difference equation
\begin{equation}
\label{eqG}
\tag{$\mathrm{eq''}_{N,r-1}$}
\begin{aligned}
\left(1-x\right)G(x) = \sum_{m=1}^{r} \sum_{j=0}^{r-1}& \left(d^{m-1} \sum_{\substack{\alpha < a(r) \\ w(\alpha)=j+m-1}} q^{-\alpha} + d^{m} \sum_{\substack{\alpha < a(r) \\ w(\alpha)=j+m}} q^{-\alpha} \right)
\\& \times {j+m-1 \brack m-1}_{q^{-N}} (-1)^{m+1} x^j q^{jmN} G\left(xq^{mN}\right).
\end{aligned}
\end{equation}
\end{lemma}
\begin{proof}
By the definition of $G$ and~\eqref{recsn}, we have
\begin{align*}
\left(1-x\right) G(x) &= \sum_{m=1}^{r} \left( d^{m-1} \sum_{\substack{\alpha < a(r) \\ w(\alpha)=m-1}} q^{-\alpha} + d^{m} \sum_{\substack{\alpha < a(r) \\ w(\alpha)=m}} q^{-\alpha} \right) (-1)^{m+1} G\left(xq^{mN}\right)
\\ &+ \sum_{m=1}^{r-1} \sum_{j=1}^{r-1}  \left(d^{m-1} \sum_{\substack{\alpha < a(r) \\ w(\alpha)=j+m-1}} q^{-\alpha} + d^{m} \sum_{\substack{\alpha < a(r) \\ w(\alpha)=j+m}} q^{-\alpha} \right)
\\& \qquad \qquad \quad \times {j+m-1 \brack m-1}_{q^{-N}} (-1)^{m+1} x^j q^{jmN} G\left(xq^{mN}\right).
\end{align*}
As the summand of the second term equals $0$ when $m=r$, we can equivalently write that the second sum is taken on $m$ going from $1$ to $r$. Then we observe that the first term corresponds to $j=0$ in the second term, and factorising gives exactly~\eqref{eqG}. Moreover, $G(0)=s_0=1$. This completes the proof.
\end{proof}

Let us do a final transformation and obtain yet another $q$-difference equation.

\begin{lemma}
\label{lemmag}
Let $G$ and $g$ be two sequences such that
$$ g(x) := G(x) \prod_{k=1}^{\infty} \left( 1 +xq^{kN-a(r)} \right).$$
Then $G$ satisfies $(\mathrm{eq}''_{N,r-1})$ and the initial condition $G(0)=1$ if and only if $g(0)=1$ and $g$ satisfies the following $q$-difference equation
\begin{equation}
\label{eqg}
\tag{$\mathrm{eq'}_{N,r-1}$}
\begin{aligned}
\left(1-x\right) g(x) &= \sum_{m=1}^{r} \left( \sum_{\nu=0}^{r-1} \sum_{\mu=0}^{\min(m-1, \nu)} f_{m,\mu} e_{m,\nu - \mu}  x^{\nu} q^{\nu m N} \right.
\\&+ \left. q^{-a(r)} \sum_{\nu=1}^{r} \sum_{\mu=0}^{\min(m-1, \nu-1)} f_{m,\mu} e_{m,\nu - \mu -1} x^{\nu} q^{\nu m N}\right) (-1)^{m+1} g\left(xq^{mN}\right),
\end{aligned}
\end{equation}
where
$$e_{m,j} := \left(d^{m-1} \sum_{\substack{\alpha < a(r) \\ w(\alpha)=j+m-1}} q^{-\alpha} + d^{m} \sum_{\substack{\alpha < a(r) \\ w(\alpha)=j+m}} q^{-\alpha} \right) {j+m-1 \brack m-1}_{q^{-N}},$$
and
$$f_{m,k} := q^{-N \frac{k(k+1)}{2} -ka(r)} {m-1 \brack k}_{q^{-N}}.$$
\end{lemma}
\begin{proof}
By definition of $g$, we have
\begin{align*}
\left(1-x\right) g(x) =& \sum_{m=1}^{r} \sum_{j=0}^{r-1} \left(d^{m-1} \sum_{\substack{\alpha < a(r) \\ w(\alpha)=j+m-1}} q^{-\alpha} + d^{m} \sum_{\substack{\alpha < a(r) \\ w(\alpha)=j+m}} q^{-\alpha} \right)
\\& \times {j+m-1 \brack m-1}_{q^{-N}} (-1)^{m+1} \prod_{k=1}^m \left(1 + xq^{kN-a(r)}\right) x^j q^{jmN} g\left(xq^{mN}\right).
\end{align*}
Furthermore
\begin{align*}
\prod_{k=1}^m \left(1 + xq^{kN-a(r)}\right) &= \prod_{k=0}^{m-1} \left(1 + xq^{(m-k)N-a(r)}\right)
\\&= \left(1 + xq^{mN-a(r)}\right) \prod_{k=1}^{m-1} \left(1 + xq^{(m-k)N-a(r)}\right)
\\&= \left(1 + xq^{mN-a(r)}\right) \sum_{k=0}^{m-1} x^k q^{kmN -N \frac{k(k+1)}{2} -ka(r)} {m-1 \brack k}_{q^{-N}},
\end{align*}
where the last equality follows from~\eqref{identity}.
Therefore
\begin{align*}
& \left(1-x\right) g(x) =
\\ & \sum_{m=1}^{r} \left(\sum_{j=0}^{r-1} \left(d^{m-1} \sum_{\substack{\alpha < a(r) \\ w(\alpha)=j+m-1}} q^{-\alpha} + d^{m} \sum_{\substack{\alpha < a(r) \\ w(\alpha)=j+m}} q^{-\alpha} \right) {j+m-1 \brack m-1}_{q^{-N}} x^j q^{jmN} \right.
\\& \left. \vphantom{\sum_{\substack{\alpha < a(r) \\ w(\alpha)=j+m-1}}}  \times  \left(1 + xq^{mN-a(r)}\right) \sum_{k=0}^{m-1} x^k q^{kmN -N \frac{k(k+1)}{2} -ka(r)} {m-1 \brack k}_{q^{-N}} \right) (-1)^{m+1} g\left(xq^{mN}\right)
\end{align*}

\begin{align*}
= \sum_{m=1}^{r} &\left[ \sum_{j=0}^{r-1} \left(d^{m-1} \sum_{\substack{\alpha < a(r) \\ w(\alpha)=j+m-1}} q^{-\alpha} + d^{m} \sum_{\substack{\alpha < a(r) \\ w(\alpha)=j+m}} q^{-\alpha} \right) {j+m-1 \brack m-1}_{q^{-N}} x^j q^{jmN} \right.
\\& \qquad \times \sum_{k=0}^{m-1} x^k q^{kmN -N \frac{k(k+1)}{2} -ka(r)} {m-1 \brack k}_{q^{-N}}
\\&+ \sum_{j=0}^{r-1} \left(d^{m-1} \sum_{\substack{\alpha < a(r) \\ w(\alpha)=j+m-1}} q^{-\alpha} + d^{m} \sum_{\substack{\alpha < a(r) \\ w(\alpha)=j+m}} q^{-\alpha} \right)
\\& \qquad \qquad \times q^{-a(r)} {j+m-1 \brack m-1}_{q^{-N}} x^{j+1} q^{(j+1)mN}
\\& \qquad  \times \left. \vphantom{\sum_{\substack{\alpha < a(r) \\ w(\alpha)=j+m-1}}} \sum_{k=0}^{m-1} x^k q^{kmN -N \frac{k(k+1)}{2}-ka(r)} {m-1 \brack k}_{q^{-N}} \right] (-1)^{m+1} g\left(xq^{mN}\right).
\end{align*}
Thus
\begin{align*}
\left(1-x\right) g(x) =& \sum_{m=1}^{r} \left(\sum_{j=0}^{r-1} e_{m,j} x^j q^{jmN} \sum_{k=0}^{m-1} f_{m,k} x^k q^{kmN} \right.
\\&+ \left. q^{-a(r)} \sum_{j=1}^{r} e_{m,j-1} x^j q^{jmN} \sum_{k=0}^{m-1} f_{m,k} x^k q^{kmN} \right) (-1)^{m+1} g\left(xq^{mN}\right).
\end{align*}
Rearranging leads to~\eqref{eqg}. As always, $g(0)=G(0)=1$. The lemma is proved.
\end{proof}

We now want to show that $f$ and $g$ are in fact equal.

\begin{lemma}
\label{equalfg}
Let $f$ and $g$ be defined as in Lemmas~\ref{lemmaf} and~\ref{lemmag}.
Then $f=g$.
\end{lemma}
\begin{proof}
To prove the equality, it is sufficient to show that for every $1 \leq m \leq r,$ the coefficient of $(-1)^{m+1} f\left(xq^{mN}\right)$ in~\eqref{eqf} is the same as the coefficient of $(-1)^{m+1}g\left(xq^{mN}\right)$ in~\eqref{eqg}.
Let $m \in \lbrace 1,...,r \rbrace$ and
\begin{align*}
S_{m} &: = \left[ (-1)^{m+1} f\left(xq^{mN}\right)\right] (\mathrm{eq}_{N,r}) 
\\&= d^{m-1} \sum_{\substack{\alpha < a(r) \\ w(\alpha)=m-1}} q^{-\alpha} + d^{m} \sum_{\substack{\alpha < a(r) \\ w(\alpha)=m}} q^{-\alpha} + \sum_{j=1}^r \sum_{k=0}^{\min(j-1,m-1)} c_{k,j} b_{m-k,j} x^j q^{jmN}
\end{align*}
and
\begin{align*}
S'_{m} &: = \left[ (-1)^{m+1} g\left(xq^{mN}\right)\right] (\mathrm{eq'}_{N,r-1}) 
\\&= \sum_{\nu=0}^{r-1} \sum_{\mu=0}^{\min(m-1, \nu)} f_{m,\mu} e_{m,\nu - \mu} x^{\nu} q^{\nu m N}
\\&\quad + q^{-a(r)} \sum_{\nu=1}^{r} \sum_{\mu=0}^{\min(m-1, \nu-1)} f_{m,\mu} e_{m,\nu - \mu -1} x^{\nu}q^{\nu m N}
\\&= f_{m,0} e_{m,0} 
\\& \quad + \sum_{\nu=1}^{r} \left( \sum_{\mu=0}^{\min(m-1, \nu)} f_{m,\mu} e_{m,\nu - \mu} + q^{-a(r)} \sum_{\mu=0}^{\min(m-1, \nu-1)} f_{m,\mu} e_{m,\nu - \mu -1} \right) x^{\nu} q^{\nu mN},
\end{align*}
because $e_{m, r-\mu}=0$ for all $\mu$, as $\mu \leq m-1$ so the sums are over $\alpha$ such that $\alpha < a(r)$ and $w(\alpha) \geq r$, which is impossible.

Let us first notice that $$f_{m,0} e_{m,0} = d^{m-1} \sum_{\substack{\alpha < a(r) \\ w(\alpha)=m-1}} q^{-\alpha} + d^{m} \sum_{\substack{\alpha < a(r) \\ w(\alpha)=m}} q^{-\alpha}.$$
Now define
$$T_{m,j}:= \sum_{k=0}^{\min(j-1,m-1)} c_{k,j} b_{m-k,j},$$
and
$$T'_{m,j}:= \sum_{k=0}^{\min(m-1, j)} f_{m,k} e_{m,j-k} + q^{-a(r)} \sum_{k=0}^{\min(m-1,j-1)} f_{m,k} e_{m,j-k-1}.$$
The only thing left to do is to show that for every $1 \leq j \leq r$, $$T_{m,j}= T'_{m,j}.$$
We have
\begin{equation}
\label{eqcb}
\begin{aligned}
&c_{k,j} b_{m-k,j} 
\\&= q^{-N \frac{k(k+1)}{2}-k a(r)} {j-1 \brack k}_{q^{-N}} {j+m-k-1 \brack m-k-1}_{q^{-N}}
\\&\quad \times  \left( d^{m-1} \sum_{\substack{ \alpha < a(r+1) \\ w(\alpha)=j+m-k-1}} q^{-\alpha} +d^m \sum_{\substack{ \alpha < a(r+1) \\ w(\alpha)=j+m-k}} q^{-\alpha} \right)
\\&= \left( d^{m-1} \sum_{\substack{ \alpha < a(r) \\ w(\alpha)=j+m-k-1}} q^{-\alpha} +d^m \sum_{\substack{ \alpha < a(r) \\ w(\alpha)=j+m-k}} q^{-\alpha} \right) q^{-N \frac{k(k+1)}{2}-k a(r)} 
\\& \quad \times {j-1 \brack k}_{q^{-N}}  {j+m-k-1 \brack m-k-1}_{q^{-N}}
\\& +q^{-a(r)} \left( d^{m-1} \sum_{\substack{ \alpha < a(r) \\ w(\alpha)=j+m-k-2}} q^{-\alpha} +d^m \sum_{\substack{ \alpha < a(r) \\ w(\alpha)=j+m-k-1}} q^{-\alpha} \right) q^{-N \frac{k(k+1)}{2}-k a(r)} 
\\& \quad \times {j-1 \brack k}_{q^{-N}}  {j+m-k-1 \brack m-k-1}_{q^{-N}},
\end{aligned}
\end{equation}
in which the last equality follows from separating the sums over $\alpha$ according to whether $\alpha$ contains $a(r)$ as a summand or not.

We also have
\begin{equation}
\label{eqfe}
\begin{aligned}
f_{m,k} e_{m,j-k} &= q^{-N \frac{k(k+1)}{2}-k a(r)} {m-1 \brack k}_{q^{-N}} 
\\&\times \left( d^{m-1} \sum_{\substack{ \alpha < a(r) \\ w(\alpha)=j+m-k-1}} q^{-\alpha} +d^m \sum_{\substack{ \alpha < a(r) \\ w(\alpha)=j+m-k}} q^{-\alpha} \right) {j+m-k-1 \brack m-1}_{q^{-N}},
\end{aligned}
\end{equation}
and
\begin{equation}
\label{eqfe'}
\begin{aligned}
q^{-a(r)} &f_{m,k} e_{m,j-k-1} = q^{-N \frac{k(k+1)}{2}-(k+1) a(r)} {m-1 \brack k}_{q^{-N}}  
\\& \times \left( d^{m-1} \sum_{\substack{ \alpha < a(r) \\ w(\alpha)=j+m-k-2}} q^{-\alpha} +d^m \sum_{\substack{ \alpha < a(r) \\ w(\alpha)=j+m-k-1}} q^{-\alpha} \right) {j+m-k-2 \brack m-1}_{q^{-N}}.
\end{aligned}
\end{equation}

By a simple calculation using the definition of $q$-binomial coefficients, we get the following result
For all $j,k,m \in \N$,
\begin{equation}
\label{equalityqbin}
{m-1 \brack k}_{q^{-N}} {j+m-k-1 \brack m-1}_{q^{-N}} = {j \brack k}_{q^{-N}} {j+m-k-1 \brack m-k-1}_{q^{-N}}.
\end{equation}
Using~\eqref{equalityqbin}, we obtain
\begin{align*}
T'_{m,j}&= \chi( j \leq m-1) \ q^{-N \frac{j(j+1)}{2}-j a(r)}  d^{m-1} {m-1 \brack m-j-1}_{q^{-N}}
\\& \qquad \times \left(\sum_{\substack{ \alpha < a(r) \\ w(\alpha)=m-1}} q^{-\alpha} +d^m \sum_{\substack{ \alpha < a(r) \\ w(\alpha)=m}} q^{-\alpha} \right)
\\&+ \sum_{k=0}^{\min(m-1,j-1)} q^{-N \frac{k(k+1)}{2}-k a(r)} {j \brack k}_{q^{-N}} {j+m-k-1 \brack m-k-1}_{q^{-N}} 
\\& \qquad \times \left( d^{m-1} \sum_{\substack{ \alpha < a(r) \\ w(\alpha)=j+m-k-1}} q^{-\alpha} +d^m \sum_{\substack{ \alpha < a(r) \\ w(\alpha)=j+m-k}} q^{-\alpha} \right)
\\&+ \sum_{k=0}^{\min(m-1,j-1)} q^{-N \frac{k(k+1)}{2}-(k+1) a(r)} {j-1 \brack k}_{q^{-N}} {j+m-k-2 \brack m-k-1}_{q^{-N}} 
\\& \qquad \times \left( d^{m-1} \sum_{\substack{ \alpha < a(r) \\ w(\alpha)=j+m-k-2}} q^{-\alpha} +d^m \sum_{\substack{ \alpha < a(r) \\ w(\alpha)=j+m-k-1}} q^{-\alpha} \right) .
\end{align*}
By~\eqref{pascal2} of Lemma~\ref{pascal}, we have
$${j \brack k}_{q^{-N}} = {j-1 \brack k}_{q^{-N}} + q^{N(k-j)} {j-1 \brack k-1}_{q^{-N}},$$
$${j+m-k-2 \brack m-k-1}_{q^{-N}} = {j+m-k-1 \brack m-k-1}_{q^{-N}} - q^{-Nj} {j+m-k-2 \brack m-k-2}_{q^{-N}}.$$
This allows us to rewrite $T'_{m,j}$ as
\begin{align*}
T'_{m,j}&= \chi( j \leq m-1) \ q^{-N \frac{j(j+1)}{2}-j a(r)} {m-1 \brack m-j-1}_{q^{-N}}
\\& \qquad \times \left( d^{m-1} \sum_{\substack{ \alpha < a(r) \\ w(\alpha)=m-1}} q^{-\alpha} +d^m \sum_{\substack{ \alpha < a(r) \\ w(\alpha)=m}} q^{-\alpha} \right)
\\&+ \sum_{k=0}^{\min(m-1,j-1)} q^{-N \frac{k(k+1)}{2}-k a(r)} {j-1 \brack k}_{q^{-N}} {j+m-k-1 \brack m-k-1}_{q^{-N}}
\\& \qquad \times \left( d^{m-1} \sum_{\substack{ \alpha < a(r) \\ w(\alpha)=j+m-k-1}} q^{-\alpha} +d^m \sum_{\substack{ \alpha < a(r) \\ w(\alpha)=j+m-k}} q^{-\alpha} \right) 
\\&+ \sum_{k=0}^{\min(m-1,j-1)} q^{-N \frac{k(k+1)}{2}-k a(r)+ N(k-j)} {j-1 \brack k-1}_{q^{-N}} {j+m-k-1 \brack m-k-1}_{q^{-N}}
\\& \qquad \times \left( d^{m-1} \sum_{\substack{ \alpha < a(r) \\ w(\alpha)=j+m-k-1}} q^{-\alpha} +d^m \sum_{\substack{ \alpha < a(r) \\ w(\alpha)=j+m-k}} q^{-\alpha} \right) 
\\&+ \sum_{k=0}^{\min(m-1,j-1)} q^{-N \frac{k(k+1)}{2}-(k+1) a(r)} {j-1 \brack k}_{q^{-N}} {j+m-k-1 \brack m-k-1}_{q^{-N}}
\\& \qquad \times \left( d^{m-1} \sum_{\substack{ \alpha < a(r) \\ w(\alpha)=j+m-k-2}} q^{-\alpha} +d^m \sum_{\substack{ \alpha < a(r) \\ w(\alpha)=j+m-k-1}} q^{-\alpha} \right) 
\\&- \sum_{k=0}^{\min(m-2,j-1)} q^{-N \frac{k(k+1)}{2}-(k+1) a(r)-Nj} {j-1 \brack k}_{q^{-N}} {j+m-k-2 \brack m-k-2}_{q^{-N}}
\\& \qquad \times \left( d^{m-1} \sum_{\substack{ \alpha < a(r) \\ w(\alpha)=j+m-k-2}} q^{-\alpha} +d^m \sum_{\substack{ \alpha < a(r) \\ w(\alpha)=j+m-k-1}} q^{-\alpha} \right) .
\end{align*}
By~\eqref{eqcb}, the sum of the second and fourth term in the sum above is exactly equal to $T{m,j}$.
Let $X$ denote the sum of the third and fifth term. We now want to show that 
\begin{align*}
X +& \chi( j \leq m-1) \ q^{-N \frac{j(j+1)}{2}-j a(r)} {m-1 \brack m-j-1}_{q^{-N}}
\\&\times \left( d^{m-1} \sum_{\substack{ \alpha < a(r) \\ w(\alpha)=m-1}} q^{-\alpha} +d^m \sum_{\substack{ \alpha < a(r) \\ w(\alpha)=m}} q^{-\alpha} \right) =0.
\end{align*}
By the change of variable $k'=k+1$ in the fourth sum, we get
\begin{align*}
X &= \sum_{k=0}^{\min(m-1,j-1)} q^{-N \frac{k(k-1)}{2}-k a(r)- Nj} {j-1 \brack k-1}_{q^{-N}} {j+m-k-1 \brack m-k-1}_{q^{-N}}
\\& \qquad \times \left( d^{m-1} \sum_{\substack{ \alpha < a(r) \\ w(\alpha)=j+m-k-1}} q^{-\alpha} +d^m \sum_{\substack{ \alpha < a(r) \\ w(\alpha)=j+m-k}} q^{-\alpha} \right) 
\\&- \sum_{k=1}^{\min(m-1,j)} q^{-N \frac{k(k-1)}{2}-k a(r)-Nj} {j-1 \brack k-1}_{q^{-N}} {j+m-k-1 \brack m-k-1}_{q^{-N}}
\\& \qquad \times \left( d^{m-1} \sum_{\substack{ \alpha < a(r) \\ w(\alpha)=j+m-k-1}} q^{-\alpha} +d^m \sum_{\substack{ \alpha < a(r) \\ w(\alpha)=j+m-k}} q^{-\alpha} \right) 
\\&= \begin{cases}
\ 0,\  &\text{if}\ j \geq m,\\
\begin{aligned} &-q^{-N \frac{j(j+1)}{2}-j a(r)} {m-1 \brack m-j-1}_{q^{-N}} \\ &\quad \times \left( d^{m-1} \sum_{\substack{ \alpha < a(r) \\ w(\alpha)=m-1}} q^{-\alpha} +d^m \sum_{\substack{ \alpha < a(r) \\ w(\alpha)=m}} q^{-\alpha} \right) ,\end{aligned}\ &\text{otherwise}
\end{cases}
\\&= -\chi( j \leq m-1) \ q^{-N \frac{j(j+1)}{2}-j a(r)} {m-1 \brack m-j-1}_{q^{-N}} 
\\& \qquad \times \left( d^{m-1} \sum_{\substack{ \alpha < a(r) \\ w(\alpha)=m-1}} q^{-\alpha} +d^m \sum_{\substack{ \alpha < a(r) \\ w(\alpha)=m}} q^{-\alpha} \right) .
\end{align*}
This completes the proof.
\end{proof}

We can finally turn to the proof of Theorem~\ref{main}.

\begin{proof}[Proof of Theorem~\ref{main}]
Let us start by the initial case $r=1$. Let $N \geq a(1)$ and $(u_m)$ such that $u_0=1$ and
\begin{equation}
\tag{$\mathrm{rec}_{N,1}$}
\left(1-dq^{\ell N-a(1)}\right) u_{\ell} = \left(1 +q^{\ell N-a(1)}\right) u_{\ell-1}.
\end{equation}
Then
\begin{displaymath}
u_{\ell}= \frac{(-q^{N-a(1)};q^N)_{\ell}}{(dq^{N-a(1)};q^N)_{\ell}}.
\end{displaymath}
Taking the limit as $\ell$ goes to infinity gives the desired result.

Now assume that Theorem~\ref{main} is true for some $r-1 \geq 1$. We want to show that it is true for $r$ too.
Let $N \geq \alpha(2^r-1)$, and $(u_m)_{m \in \N}$ satisfying $(\mathrm{rec}_{N,r})$ and the initial condition $u_0=1$.
For all $m$, let
\begin{displaymath}
\beta_m:= u_m \prod_{j=1}^{m} \frac{1-dq^{jN-a(r)}}{1-q^{jN}}.
\end{displaymath}
Then $\beta_0=1$ and by Lemma~\ref{lemmau}, $(\beta_m)$ satisfies~\eqref{recbeta}.
Now let $$f(x):= \sum_{n=0}^{\infty} \beta_n x^n.$$
Then by Lemma~\ref{lemmaf}, $f(0)=1$ and $f$ satisfies~\eqref{eqf}.
But by Lemma~\ref{equalfg}, $f$ also satisfies~\eqref{eqg}.
Now let  $$ G(x):= \frac{f(x)}{\prod_{k=1}^{\infty} \left(1 + xq^{kN-a(r)}\right)}.$$
By Lemma~\ref{lemmag}, $G(0)=1$ and $G$ satisfies~\eqref{eqG}.
Let $$G(x) =: \sum_{n=0}^{\infty} s_n x^n.$$
By Lemma~\ref{lemmaG}, $s_0=1$ and $(s_n)$ satisfies~\eqref{recsn}.
Finally let $$\mu_n:= s_n \prod_{k=1}^{n} \left(1-q^{Nk}\right).$$
By Lemma~\ref{lemmasn}, $\mu_0=1$ and $(\mu_n)$ satisfies $(\mathrm{rec}_{N,r-1})$.
Now $N$ is still larger than $\alpha \left( 2^{r-1}-1 \right)$ and we can use the induction hypothesis which gives
\begin{equation}
\label{g1}
\lim\limits_{\ell \rightarrow \infty}\mu_{\ell} = \prod_{k=1}^{r-1} \frac{(-q^{N-a(k)};q^N)_{\infty}}{(dq^{N-a(k)};q^N)_{\infty}}..
\end{equation}
Therefore by definition of $(s_{\ell})$,
\begin{displaymath}
\lim\limits_{\ell \rightarrow \infty}s_{\ell} = \frac{1}{(q^N;q^N)_{\infty}} \prod_{k=1}^{r-1} \frac{(-q^{N-a(k)};q^N)_{\infty}}{(dq^{N-a(k)};q^N)_{\infty}}.
\end{displaymath}
We have
\begin{equation}
\label{truc}
\sum_{m=0}^{\infty} \beta_m x^m = f(x) = \prod_{k=1}^{\infty} \left(1 + xq^{kN-a(r)}\right) G(x) = \prod_{k=1}^{\infty} \left(1 + xq^{kN-a(r)}\right) \sum_{m=0}^{\infty} s_m x^m.
\end{equation}
We multiply both sides of \eqref{truc} by $(1-x)$ and we apply Appell's Comparison Theorem~\cite[p. 101]{Appell}. We obtain
\begin{displaymath}
\lim\limits_{\ell \rightarrow \infty} \beta_{\ell}= \prod_{k=1}^{\infty} \left(1 + q^{kN-a(r)}\right) \lim\limits_{\ell \rightarrow \infty}s_{\ell} = \frac{(-q^{N-a(r)};q^N)_{\infty}}{(q^N;q^N)_{\infty}} \prod_{k=1}^{r-1}\frac{(-q^{N-a(k)};q^N)_{\infty}}{(dq^{N-a(k)};q^N)_{\infty}}.
\end{displaymath}
Thus by definition of $(\beta_{\ell})$, we have
\begin{displaymath}
\lim\limits_{\ell \rightarrow \infty} u_{\ell}=\prod_{j=1}^{\infty} \frac{1-q^{jN}}{1-dq^{jN-a(r)}} \lim\limits_{\ell \rightarrow \infty}\beta_{\ell} = \frac{(-q^{N-a(r)};q^N)_{\infty}}{(dq^{N-a(r)};q^N)_{\infty}} \prod_{k=1}^{r-1}\frac{(-q^{N-a(k)};q^N)_{\infty}}{(dq^{N-a(k)};q^N)_{\infty}}.
\end{displaymath}
Theorem~\ref{main} is proved.
\end{proof}

Now Theorem~\ref{dousse2} is a simple corollary of Theorem~\ref{main}.

\begin{proof}[Proof of Theorem~\ref{dousse2}]
We have that $\lim\limits_{\ell \rightarrow \infty} u_{\ell}$, which is the generating function for partitions counted by $G(-A'_N;k,n)$, is equal to $\prod_{k=1}^r \frac{(-q^{N-a(k)};q^N)_{\infty}}{(dq^{N-a(k)};q^N)_{\infty}},$ which is the generating function for partitions counted by $F(-A_N;k,n)$.
Thus $$F(-A_N;n,k)= G(-A'_N;n,k),$$ and Theorem~\ref{dousse2} is proved.
\end{proof}

\section{Conclusion}
In~\cite{Corteel}, Corteel and Lovejoy proved an even more general theorem of which both of Andrews' theorems are particular cases. It would be interesting to generalise it to overpartitions too, but new techniques might be necessary.

It would also be interesting to see if Theorems~\ref{dousse2} and~\ref{dousse} have connections with representation theory and quantum algebra like Theorems~\ref{andrews2} and~\ref{andrews}.

\section*{Acknowledgements}
The author would like to thank Jeremy Lovejoy for carefully reading a preliminary version of this article and giving her helpful advice to improve it.

\bibliographystyle{alpha}

\bibliography{references}   

\end{document}